\documentclass[12pt]{amsart}
\usepackage{amsmath,amssymb,amsbsy,amsfonts,latexsym,amsopn,amstext,cite,
                                               amsxtra,euscript,amscd,bm,mathabx}
\usepackage{url}
\usepackage[colorlinks,linkcolor=blue,anchorcolor=blue,citecolor=blue,backref=page]{hyperref}
\usepackage{color}
\usepackage{graphics,epsfig}
\usepackage{graphicx}
\usepackage{float}

\hypersetup{breaklinks=true}

\usepackage[norefs,nocites]{refcheck}

\begin{document}

\newtheorem{theorem}{Theorem}
\newtheorem{lemma}[theorem]{Lemma}
\newtheorem{claim}[theorem]{Claim}
\newtheorem{cor}[theorem]{Corollary}
\newtheorem{prop}[theorem]{Proposition}
\newtheorem{definition}{Definition}
\newtheorem{question}[theorem]{Question}
\newcommand{\hh}{{{\mathrm h}}}

\numberwithin{equation}{section}
\numberwithin{theorem}{section}
\numberwithin{table}{section}

\def\sssum{\mathop{\sum\!\sum\!\sum}}
\def\ssum{\mathop{\sum\ldots \sum}}
\def\iint{\mathop{\int\ldots \int}}

\newcommand{\supp}{\operatorname{supp}}

\def\squareforqed{\hbox{\rlap{$\sqcap$}$\sqcup$}}
\def\qed{\ifmmode\squareforqed\else{\unskip\nobreak\hfil
\penalty50\hskip1em \nobreak\hfil\squareforqed
\parfillskip=0pt\finalhyphendemerits=0\endgraf}\fi}%%

%  use the AMS-Euler Fraktur fonts
%%%%%%%%%%%%%%%%%%%%%%%%%%%%%%%%%%
\newfont{\teneufm}{eufm10}
\newfont{\seveneufm}{eufm7}
\newfont{\fiveeufm}{eufm5}
%%%%%%%%%%%%%%%%%%%%%%%%%%%%%%%%%
%
%  allow automatic size selection in math mode
%
%%%%%%%%%%%%%%%%%%%%%%%%%%%%%%%%%
\newfam\eufmfam
     \textfont\eufmfam=\teneufm
\scriptfont\eufmfam=\seveneufm
     \scriptscriptfont\eufmfam=\fiveeufm
%%%%%%%%%%%%%%%%%%%%%%%%%%%%%%%%%
%
%  \frak works on a single symbol at a time...
%
\def\frak#1{{\fam\eufmfam\relax#1}}

\newcommand{\bflambda}{{\boldsymbol{\lambda}}}
\newcommand{\bfmu}{{\boldsymbol{\mu}}}
\newcommand{\bfxi}{{\boldsymbol{\xi}}}
\newcommand{\bfrho}{{\boldsymbol{\rho}}}

\def\eps{\varepsilon}

\def\fK{\mathfrak K}
\def\fT{\mathfrak{T}}

\def\fA{{\mathfrak A}}
\def\fB{{\mathfrak B}}
\def\fC{{\mathfrak C}}
\def\fM{{\mathfrak M}}
\def\fS{{\mathfrak  S}}
\def\fU{{\mathfrak U}}
\def\DAR{\cD(X;\cI,q)}
\def\DARp{\cD(X;\cI,p)}
\def\EAR{\cE(X;\cI,q)}
\def\EARp{\cE(X;\cI,p)}
\def\DaR{\cD^*(X;\cI,q)}
\def\EaR{\cE^*(X;\cI,q)}
\def\DaRp{\cD^*(X;\cI,p)}

\def\DAQ{\mathbf{D}(X;\cA, \cQ)}
\def\EAQ{\mathbf{E}(X;\cA, \cQ)}

\def\SXaq{S(X;a,q)}
\def\MXaq{\mathrm{M}(X;a,q)}
\def\RXaq{\mathrm{R}(X;a,q)}
\def\RXap{\mathrm{R}(X;a,p)}

\def\Kmnd{\cK_d(m,n)}
\def\Kmnp{\cK_p(m,n)}
\def\Kmnq{\cK_q(m,n)}

\def\SanIJd{\cS_d(\balpha, \bnu;\cI,\cJ)}
\def\SnIJd{\cS_d(\bnu;\cI,\cJ)}
\def\SIJd{\cS_d(\cI,\cJ)}
\def\SanIJp{\cS_p(\balpha, \bnu;\cI,\cJ)}
\def\SnIJp{\cS_p(\bnu;\cI,\cJ)}
\def\SIJp{\cS_p(\cI,\cJ)}
\def\SAMJp{\cS_p(\balpha;\cM,\cJ)}
\def\SAMJq{\cS_q(\balpha;\cM,\cJ)}
\def\SAJq{\cS_q(\balpha;\cJ)}
\def\SAIJp{\cS_p(\balpha;\cI,\cJ)}
\def\SAIJq{\cS_q(\balpha;\cI,\cJ)}

\def \balpha{\bm{\alpha}}
\def \bbeta{\bm{\beta}}
\def \bgamma{\bm{\gamma}}
\def \blambda{\bm{\lambda}}
\def \bchi{\bm{\chi}}
\def \bphi{\bm{\varphi}}
\def \bpsi{\bm{\psi}}
\def \bnu{\bm{\nu}}
\def \bomega{\bm{\omega}}

\def \bxi{\bm{\xi}}

\def\eqref#1{(\ref{#1})}

\def\vec#1{\mathbf{#1}}

\newcommand{\abs}[1]{\left| #1 \right|}

\def\Zq{\mathbb{Z}_q}
\def\Zqx{\mathbb{Z}_q^*}
\def\Zd{\mathbb{Z}_d}
\def\Zdx{\mathbb{Z}_d^*}
\def\Zf{\mathbb{Z}_f}
\def\Zfx{\mathbb{Z}_f^*}
\def\Zp{\mathbb{Z}_p}
\def\Zpx{\mathbb{Z}_p^*}
\def\M{\mathcal M}
\def\E{\mathcal E}
\def\cH{\mathcal H}
%\def\squareforqed{\hbox{\rlap{$\sqcap$}$\sqcup$}}
%\def\qed{\ifmmode\squareforqed\else{\unskip\nobreak\hfil
%\penalty50\hskip1emrll\nobreak\hfil\squareforqed
%\parfillskip=0pt\finalhyphendemerits=0\endgraf}\fi}

%%%%%%%%%%%%%%%%%%%%%%%%%
% Alphabet calligraphie %
%%%%%%%%%%%%%%%%%%%%%%%%%
\def\cA{{\mathcal A}}
\def\cB{{\mathcal B}}
\def\cC{{\mathcal C}}
\def\cD{{\mathcal D}}
\def\cE{{\mathcal E}}
\def\cF{{\mathcal F}}
\def\cG{{\mathcal G}}
\def\cH{{\mathcal H}}
\def\cI{{\mathcal I}}
\def\cJ{{\mathcal J}}
\def\cK{{\mathcal K}}
\def\cL{{\mathcal L}}
\def\cM{{\mathcal M}}
\def\cN{{\mathcal N}}
\def\cO{{\mathcal O}}
\def\cP{{\mathcal P}}
\def\cQ{{\mathcal Q}}
\def\cR{{\mathcal R}}
\def\cS{{\mathcal S}}
\def\cT{{\mathcal T}}
\def\cU{{\mathcal U}}
\def\cV{{\mathcal V}}
\def\cW{{\mathcal W}}
\def\cX{{\mathcal X}}
\def\cY{{\mathcal Y}}
\def\cZ{{\mathcal Z}}
\newcommand{\rmod}[1]{\: \mbox{mod} \: #1}

\def\cg{{\mathcal g}}

\def\vr{\mathbf r}

\def\e{{\mathbf{\,e}}}
\def\ep{{\mathbf{\,e}}_p}
\def\eq{{\mathbf{\,e}}_q}

\def\Tr{{\mathrm{Tr}}}
\def\Nm{{\mathrm{Nm}}}

 \def\SS{{\mathbf{S}}}

\def\lcm{{\mathrm{lcm}}}

\def\({\left(}
\def\){\right)}
\def\l|{\left|}
\def\r|{\right|}
\def\fl#1{\left\lfloor#1\right\rfloor}
\def\rf#1{\left\lceil#1\right\rceil}
\def\sumstar#1{\mathop{\sum\vphantom|^{\!\!*}\,}_{#1}}

\def\mand{\qquad \mbox{and} \qquad}

\newcommand{\commI}[1]{\marginpar{%
\begin{color}{magenta}
\vskip-\baselineskip %raise the marginpar a bit
\raggedright\footnotesize
\itshape\hrule \smallskip I: #1\par\smallskip\hrule\end{color}}}

\newcommand{\commB}[1]{\marginpar{%
\begin{color}{blue}
\vskip-\baselineskip %raise the marginpar a bit
\raggedright\footnotesize
\itshape\hrule \smallskip B: #1\par\smallskip\hrule\end{color}}}

%%%%%%%%%%%%%%%%%%%%%%%%%%%%%%%%%%%%%%%%%%%%%%%%%%%%%%%%
%%%%%%%%%%%%%%%%%%%%%%%%%%%%%%%%%%%%%%%%%%%%%%%%%%%%%%%%
%%%%%%%%%%%%%%%%%%%%%%%%%%%%%%%%%%%%%%%%%%%%%%%%%%%%%%%%
%%%%%%%%%%%%%%%%%%%%%%%%%%%%%%%%%%%%%%%%%%%%%%%%%%%%%%%%

%%%%%%%  END OF STANDARD STUFF %%%%%%%%%

%%%%%%%%%%%%%%%%%%%%%%%%%%%%%%%%%%%%%%%%%%%%%%%%%%%%%%%%
%%%%%%%%%%%%%%%%%%%%%%%%%%%%%%%%%%%%%%%%%%%%%%%%%%%%%%%%
%%%%%%%%%%%%%%%%%%%%%%%%%%%%%%%%%%%%%%%%%%%%%%%%%%%%%%%%
%%%%%%%%%%%%%%%%%%%%%%%%%%%%%%%%%%%%%%%%%%%%%%%%%%%%%%%
%%%%%%%%%%%
%%% Spell

\hyphenation{re-pub-lished}

\mathsurround=1pt

\def\bfdefault{b}

\def \F{{\mathbb F}}
\def \K{{\mathbb K}}
\def \Z{{\mathbb Z}}
\def \Q{{\mathbb Q}}
\def \R{{\mathbb R}}
\def \C{{\mathbb C}}
\def\Fp{\F_p}
\def \fp{\Fp^*}

 \def \xbar{\overline x}

%\title[Divisor function in  arithmetic progressions]{Bilinear sums of Kloosterman sums and average values of  the  divisor function over families of  arithmetic progressions}

\title[Divisor function in  arithmetic progressions]{Bilinear sums of Kloosterman sums, 
multiplicative congruences and average values of  the  divisor function over families of  arithmetic progressions}

 \author[B. Kerr] {Bryce Kerr}
\address{Department of Pure Mathematics, University of New South Wales,
Sydney, NSW 2052, Australia}
\email{bryce.kerr89@gmail.com}

 \author[I. E. Shparlinski] {Igor E. Shparlinski}
 
\address{Department of Pure Mathematics, University of New South Wales,
Sydney, NSW 2052, Australia}
\email{igor.shparlinski@unsw.edu.au}

\begin{abstract} 
We obtain several  asymptotic formulas  for the sum of the divisor function $\tau(n)$ with $n \le x$ in  an arithmetic progressions   $n \equiv a \pmod q$  on average over $a$ from a set of several consecutive elements  from set of reduced residues modulo $q$ and on average over arbitrary sets.
The main goal is to obtain nontrivial result for $q \ge x^{2/3}$ with the small amount of averaging over $a$. 
We recall that for individual values of $a$ the limit  of our current methods is $q  \le x^{2/3-\varepsilon}$ 
for an arbitrary fixed $\varepsilon> 0$.  Our method builds on an approach due to Blomer~(2008) based on the Voronoi summation formula which we combine with some 
recent results on bilinear sums of  Kloosterman sums due Kowalski, Michel and 
Sawin~(2017) and Shparlinski~(2017). We also make use of extra applications of the Voronoi summation formulae after expanding into Kloosterman sums and this reduces the problem to estimating the number of solutions to multiplicative congruences. 
\end{abstract}


\maketitle

\section{Introduction}
%\label{sec intro}

\subsection{Background}
Let 
$$\tau(n)=\sum_{d\mid n}1,$$
 denote the   {\it divisor function\/},  
 where the sum runs over all positive integral divisors $d$ of  an integer $n \ge 1$.
 
For integers $a$ and $q\ge 2$ with $\gcd(a,q)=1$, consider the
{\it divisor sum\/} given by:
$$
\SXaq=\sum_{\substack{n\le X\\n\equiv a\pmod q}}\tau(n).
$$

Several authors proved independently an asymptotic formula for 
$S(X,a,q)$ in the range $q\le X^{2/3-\varepsilon}$ with an arbitrary 
fixed $\varepsilon>0$, see discussions and proofs in~\cite{Blom,Hool,PoVau}.

To formulate these results more precisely, we need to introduce some notation. 
Namely, we define the  polynomial
$$
P(T;q, a) = \sum_{d\mid q} \frac{r_d(a)}{d}\(T -2 \log d +2\gamma-1\), 
$$
where $\gamma$ is the {\it Euler-Mascheroni constant\/}
$$
r_d(a) = \sum_{e\mid\gcd(a,d)} e \mu(d/e),
$$
is the {\it Ramanujan sum\/}, 
and  $\mu(k)$ is the  {\it M{\"o}bius function\/}.

We now define
$$
\MXaq  =  \frac{X}{q}  P(\log X;q, a), 
$$
which is the expected main term in the asymptotic formula for the sum $\SXaq$, and  thus we also 
define  the error term
$$
\RXaq = \SXaq  - \MXaq.
$$
It is useful to note that if $\gcd(a,q)=1$ then 
$$
\MXaq= 
\frac{\varphi(q) }{q^2}X \(\ln X+2\gamma-1\)
- \frac{2}q X\sum_{d\,|\,q}\frac{\mu(d)\ln d}d, 
$$
where  $\varphi(k)$ is the {\it Euler function\/}. 

Then,  uniformly over integers $a$ with $\gcd(a,q)=1$ we have the bound
\begin{equation}
\label{eq: Ind Bound}
\RXaq  \le  X^{1/3+o(1)}, 
\end{equation}
which given by Blomer~\cite[Equation~(2)]{Blom} (see also~\cite{FIK}) and  generalised to the case of arbitrary $\gcd(a,q)$ 
by Pongsriiam and  Vaughan~\cite[Theorem~1.1]{PoVau}.

Furthermore,  Blomer~\cite[Theorem~1.1]{Blom}, improving the previous result of
Banks,   Heath-Brown and  Shparlinski~\cite{BHBS}, has shown that 
\begin{equation}
\label{eq:  Av Bound a}
\sum_{a=0}^{q-1} \RXaq^2 \le X^{1+o(1)}, 
\end{equation}
which (as also the result of~\cite{BHBS}) is nontrivial in the essentially optimal range  
$q\le X^{1-\varepsilon}$ with an arbitrary fixed $\varepsilon>0$. 

With respect to a different kind of averaging, namely, over $q$ rather than over $a$, 
Fouvry~\cite[Corollary~5]{Fouv} has obtained the following bound: for any fixed $\varepsilon>0$
there exists some constant $c >0$ such that uniformly over integers $a$ with 
$|a|  \le \exp\(c (\log X )^{1/2}\)$ we have
\begin{equation}
\label{eq:  Av Bound q}
\sum_{\substack{X^{2/3+\varepsilon} \le q \le X^{1-\varepsilon}\\ \gcd(a,q)=1}} 
\left| \RXaq \right| =  O\( X \exp\(-c (\log X)^{1/2}\)\).
\end{equation}

We note that the summation in~\eqref{eq: Av Bound q}  can be extended to $q \le X^{2/3-\varepsilon}$, 
however the values of $q$ in the range $X^{2/3-\varepsilon} < q  <  X^{2/3+\varepsilon}$ have to be avoided. For a class of special moduli,  this gap in the range of $q$ has been bridged in~\cite{FIK}. 

\subsection{New set-up and results}

 Here we consider two apparently new questions,  which ``interpolate'' between obtaining individual bounds
 like~\eqref{eq: Ind Bound} and   bounds on average like~\eqref{eq: Av Bound a}. 
 Namely, given some subset $\cA\subseteq \Z_q^{*}$ of the reduced residues modulo $q$ we consider the sums
$$
 \cD(X;\cA,q) =  \sum_{a \in \cA} \left| R(X;a,q) \right| \quad \text{and}\quad
  \cE(X;\cA,q) =  \sum_{a \in \cA} R(X;a,q). 
$$
In particular, using~\eqref{eq:  Av Bound a} and the Cauchy-Schwarz inequality, we obtain
\begin{equation}
\label{eq: Cauchy}
| \cE(X;\cA,q)| \le \cD(X;\cA,q) \le A^{1/2}  X^{1/2+o(1)}, 
\end{equation}
which is nontrivial (that is, stronger than the trivial upper bound $AXq^{-1+o(1)}$) provided that $A\ge q^{2+\eps} X^{-1}$
for some fixed $\eps > 0$. Thus in the case $q \sim X^{2/3}$, this becomes 
 $A\ge q^{1/2+\eps}$.
 
 We are  interested in obtaining stronger bounds on $\cD(X;\cA,q)$ and $\cD(X;\cA,q)$ 
and especially which are 
nontrivial for  small values of $A$  and large values of $q$
(for example, when $A\le q^{1/2}$ and $q \ge X^{2/3}$).  

Our first bound considers the case when $\cA=\cI$ is an interval and depends on a result of  Kowalski, Michel and 
Sawin~\cite[Theorem~1.1]{KMS1},  and thus applies only to 
prime $q=p$.  In particular, as we are mostly interested in the values $q =p \ge X^{2/3}$, to simplify the  result we assume that $p \ge X^{4/7}$.  
 
 \begin{theorem}
\label{thm:DAR}  
For any integers $A$ and $X$, an interval $\cI$ of length $A$ and a prime $p$ with 
$$ A\le p  \mand X \ge p \ge X^{4/7},
$$ 
we have, 
\begin{align*}
\DARp  & \le  (AX^{1/2} p^{-1/2}     +A^{3/2} X^{1/2}  p^{-5/8} \\
&\qquad \qquad + A^{1/2} X^{1/2}  p^{-1/8} +  A^{5/6} X^{5/18} p^{11/72})p^{o(1)}. 
\end{align*}
\end{theorem}

We now see that the bound of Theorem~\ref{thm:DAR} is nontrivial, 
that is, better than $AX/p$,  for 
$$
X  p^{-3/4-\varepsilon} \ge A \ge  
\max\left\{ X^{-13/3} p^{83/12+ \varepsilon}, X^{-1} p^{7/4+ \varepsilon}\right\} \quad 
\text{and} \quad
p \le X^{1-\varepsilon}, 
$$
for some fixed $\varepsilon > 0$. In particular at the critical value $p \sim X^{2/3}$
this condition becomes
$$
 p^{3/4-\varepsilon} \ge A \ge p^{5/12 + \varepsilon}.
$$

Clearly,    $ |\EAR| \le \DAR$, see~\eqref{eq: Cauchy}, 
 but in general we obtain  a stronger result for $\EAR$ which  does not follow from this trivial inequality and which also applies to composite moduli. 
 
 Again, as we are mostly interested in the values $q \ge X^{2/3}$, we make a 
simplifying  assumtion that $q \ge X^{19/31}$ (and note that $19/31 < 2/3$).

  \begin{theorem}
\label{thm:EAR} 
For any integers $A$, $X$ and $q$ with 
$$ A\le q  \mand X \ge q \ge X^{19/31},
$$ 
and any interval $\cI$ of length $A$ we have,
$$
|\EAR|  \le   \( A X^{1/2}   q^{-1/2 } + A^{1/8}  X^{1/4} q^{1/2} + A^{1/2}  X^{1/4} q^{1/4}  
\) q^{o(1)}.
$$
\end{theorem}

We now see that the bound of Theorem~\ref{thm:EAR} is nontrivial, 
that is, better that $AX/q$ for 
$$
A \ge \max\{  X^{-6/7} q^{12/7+ \varepsilon}, X^{-3/2}  q^{5/2+ \varepsilon}  \} \mand q \le X^{1-\varepsilon}, 
$$
for some fixed $\varepsilon > 0$. In particular, at  the critical value $q \sim X^{2/3}$
this condition becomes
$$
A \ge q^{3/7 + \varepsilon}.
$$
 
The proofs of the above estimates are based on an approach of Blomer~\cite{Blom} combined with some recent bounds on bilinear 
 sums of Kloosterman sums, see~\cite{BFKMM,FKM,KMS1,KMS2,Shp,ShpZha,Xi-FKM} for a variety 
 of such bounds.

 Our last result considers averaging over an arbitrary set $\cA$ and uses extra applications of the Voronoi summation formula to 
 reduce to estimating solutions to multiplicative congruences rather than Kloosterman sums. We use $\F_p$ to denote the field
 of $p$ elements. 
 
 \begin{theorem}
\label{thm:DAR-small} 
Let $p$ be prime and $\cA\subseteq \F_p^{*}$ be any set with cardinality $A$.
 For any integer $X\ge 1$ we have  
$$
 \cD(X;\cA,p) \le   
 \(A^{3/4}X^{1/4}p^{1/4}+A^{2/3}X^{1/3}\)X^{o(1)}. 
$$
\end{theorem}

In particular, we see that Theorem~\ref{thm:DAR-small} provides an improvement over~\eqref{eq: Ind Bound} with trivial summation over $A$, 
that is, over $AX^{1/3+o(1)}$ once the condition
$$
p \le \min\{ A X^{1/3-\varepsilon}, X^{2/3-\varepsilon}\}
$$ 
is satisfied. 
Furthermore, it improves the bound~\eqref{eq: Cauchy} once 
$$
p \le X^{1-\varepsilon}/A. 
$$
Hence the above two conditions are equivalent to
$$
p \le \min\{ A X^{1/3-\varepsilon}, X^{1-\varepsilon}/A\}.
$$ 

We can now estimate the set of $a$ for which the bound~\eqref{eq: Ind Bound}
is almost tight. Namely for a fixed $\kappa>0$ we denote by $\cA_\kappa(X,p)$ the set 
of $a\in \F_p^*$ for which $\RXap \ge  X^{1/3 -\kappa}$. 

Since  $\cD(X;\cA_\kappa(X,p) ,p) \ge \# \cA_\kappa(X,p) X^{1/3 -\kappa}$, Theorem~\ref{thm:DAR-small} 
yields: 

 \begin{cor}
\label{cor:large R} 
Let $p$ be prime and $\cA\subseteq \F_p^{*}$ be any set with cardinality $A$.
 For any integer $X\ge 1$ we have  
$$
 \# \cA_\kappa(X,p) \le   \max\{ p  X^{-1/3+4 \kappa} ,  X^{3 \kappa}\}X^{o(1)}. 
$$
\end{cor}

\section{Preliminaries}

\subsection{Notation}
 Let  $q$ be  a positive 
integer. We denote the residue ring modulo $q$ by $\Z_q$ and denote the group
of units of $\Z_q$ by $\Z_q^*$. 

 For integers $d\ge 1$, $m$ and $n$ we define
the {\it Kloosterman sum\/}
$$
\Kmnd = \sum_{x \in \Z_q^*} \e_d\(mx +n\xbar \),
$$
where $\xbar$ is the multiplicative inverse of $x$ modulo $q$ and
$$
\e_d(z) = \exp(2 \pi i z/d).
$$

We use $\supp F$ to denote the support  of a real valued function $f$, that is, 
we have $F(x) \ne 0$ if and only if $x \in \supp F$. 

As usual, $\{ x\}$ denotes the fractional part of a real number $x$.

 We recall that the expressions
$U \ll V$, $V \gg U$ and $U=O(V)$ are all equivalent to the statement
that $|U|\le cV$ for some constant $c$. Throughout the paper, the implied constants in symbols ``$O$'',
``$\ll$'' and  ``$\gg$'' may occasionally, where obvious, depend on the
small positive parameter $\eps$ and integer parameters $k$ and $r$, and are absolute otherwise.

\subsection{Error terms and Kloosterman sums}

In this section, we collect some useful results that stem from the work of 
Blomer~\cite{Blom} and link bounds on $ \DAR$  and $\EAR$ to bounds 
on some  bilinear sums of Kloosterman sums. 

Fix some parameter 
\begin{equation}
\label{eq:Ypar}
1\le Y \le X/2
\end{equation}
 and a smooth compactly supported function $w(x)$ satisfying
$$
w(x)=\begin{cases}1& \text{if\ }  x\in [2Y,X], \\ 
0 & \text{if\ }  x\le Y \ \text{or \ }  x\ge X+Y, \end{cases}
$$
and for any integer $j\ge 1$
$$
w^{(j)}(x)\ll \frac{1}{Y^j}.
$$

We now recall the link between the error term $\RXaq$ and some bilinear sums of Kloosterman sums
given by  Blomer~\cite{Blom}. Namely by~\cite[Equations~(7) and~(8)]{Blom}, for any fixed $\eps > 0$ 
and a parameter 
$Y$ satisfying~\eqref{eq:Ypar},  we have
\begin{equation}
\begin{split}
\label{eq:RXaqComplete}
\RXaq  = \frac{1}{q} \sum_{d \mid q} \sum_\pm \sum_{n=1}^\infty \tau(n)  u_d^\pm(n) & \cK_d(\mp n, a)\\
& + O\((Y/q + 1) (Yq)^{\eps}\), 
\end{split}
\end{equation}
with the functions $u_d^{\pm}$ defined by 
$$
u_d^+(y)=\frac{4}{d}\int_{-\infty}^{\infty}w(x)K_0\left(\frac{4\pi (xy)^{1/2}}{d}\right)dx,
$$
and  
$$
u_d^-(y)=-\frac{2\pi}{d}\int_{-\infty}^{\infty}w(x)Y_0\left(\frac{4\pi (xy)^{1/2}}{d}\right)dx,
$$
where we use the standard notation for Bessel functions $K_0(x)$ and $Y_0(x)$
(we note that for typographical simplicity we have replaced the sum 
$\cK_d(\pm n, -a)$ with the equal sum $\cK_d(\mp n, a)$). 

We note that  in~\cite{Blom} a  slightly more complicated notation 
$\widecheck{w}^\pm(n)$  is used, however with the dependence on $d$ suppressed.  
In turn, we also suppress the dependence on $X$ in 
the notation for the functions $u_d^\pm(y)$, as well as we do for the following
quantities 
\begin{equation}
\label{eq: Ud Vd}
U(d) = d^2 X^{-1}  \mand  V(d) = d^2 X^{1+\eps} Y^{-2}. 
\end{equation}

We now recall,   by~\cite[Equations~(11)]{Blom}  we have  
\begin{equation}
\label{eq: u size}
u_d^\pm(n)  \ll 
 \begin{cases} X^{1+\eps} d^{-1}, & \text{if\ } n \le U(d), \\
 X^{1/4}  d^{1/2} n^{-3/4}, & \text{if\ }  U(d) < n \le V(d).
 \end{cases}
\end{equation}
Furthermore, 
\begin{equation}
\label{eq:ur3}
u_d^\pm(n) \ll \frac{1}{n^{c}}, \quad \text{if\ }   n>V(d),
\end{equation}
for any fixed   $c>0$, with the implied constant depending on $c$. 
In particular,  the contribution to $\RXaq$ from  $n \ge V(d)$ is negligible
and  we can limit the summation over $n$ up to $V(d)$ and absorb 
the difference in the already present error term.
We then have 
\begin{equation}
\begin{split}
\label{eq:Rerror}
R(X;a,q)= \frac{1}{q}\sum_{d \mid q}  \sum_{\pm}  \sum_{n=1}^{V(d)} &\tau(n)  u_d^\pm(n)  \cK_d(\mp n, a)\\
 &\qquad + O\((Y/q + 1) (Yq)^{\eps}\).
\end {split}
\end{equation}

Hence, changing the order of summations,  we derive from~\eqref{eq:Rerror}  that
\begin{equation}
\label{eq: DAB Dab}
 \DAR \le  \DaR + O\(A(Y/q + 1)(Yq)^{\eps}\),
 \end{equation}
 where 
$$
\DaR =  \frac{1}{q} \sum_{d \mid q} \sum_\pm  \sum_{a \in \cI}  
\left| \sum_{n=1}^{V(d)} \tau(n)  u_d^\pm(n)  \cK_d(\mp n, a)\right|.
$$
Hence  for some complex numbers $\alpha_{a,d}$ with $|\alpha_{a,d}|=1$  we
can write 
\begin{equation}
\label{eq: DaR}
\DaR =  \frac{1}{q} \sum_{d \mid q} \sum_\pm    \sum_{a \in \cI}  
\sum_{n=1}^{V(d)} \alpha_{a,d} \tau(n)  u_d^\pm(n)  \cK_d(\mp n, a).  
 \end{equation}
 
 For $\EAR$ there is no need to introduce the weights $\alpha_{a,d}$, 
 so we have 
 \begin{equation}
\label{eq: EAB Eab}
 |\EAR| \le  \EaR + O\(A(Y/q + 1)(Yq)^{\eps}\),
 \end{equation}
 where 
\begin{equation}
\label{eq: EaR}
\EaR =  \frac{1}{q} \sum_{d \mid q} \sum_\pm  \left|  \sum_{a \in \cI}  
\sum_{n=1}^{V(d)}  \tau(n)  u_d^\pm(n)  \cK_d(\mp n, a)\right|. 
 \end{equation}

\subsection{Bilinear sums of Kloosterman sums}
%\label{sec Kloost}

Given  two intervals
$$ 
\cI= \{B+1, \ldots,  B+A\}, \ \cJ =  \{M+1,   \ldots,  M+N\}\subseteq [1, q-1]
$$
of $A$ and  $N$ consecutive integers, respectively
and two  sequence of weights $\balpha = \{\alpha_a\}_{a\in \cI}$ and $\bnu = \{\nu_n\}_{n\in \cJ}$,
we define the following bilinear sums of Kloosterman sums
\begin{align*}
&\SanIJd = \sum_{a\in \cI} \sum_{n \in \cJ} \alpha_a   \nu_n \cK_d(n,a),\\
&\SnIJd  = \sum_{a\in \cI} \sum_{n \in \cJ} \nu_n \cK_d(n,a).
\end{align*}

We now collect some bounds on these sums  
slightly simplifying and adjusting them to our notation; more bounds can be found
in~\cite{BFKMM,FKM,KMS1,KMS2,Shp,ShpZha, Xi-FKM}.

For  general bilinear sums, we recall a bound of  Kowalski, Michel and 
Sawin~\cite[Theorem~1.1]{KMS1}, which improves that  
of Fouvry,  Kowalski and Michel~\cite[Theorem~1.17]{FKM}, see also~\cite[Theorem~5.1]{BFKMM} near $A\sim N \sim p^{1/2}$, see also~\cite[Theorem~5.1]{BFKMM}, which however is known only 
for prime moduli $q = p$.

\begin{lemma}
\label{lem:SanIJp} For a prime $p \ge 1$, integers  $1 \le A,N\le p$, an initial interval  
$\cJ = \{1,  \ldots, N\}$, weights $\balpha$ and $\bnu$ with
$$
|\alpha_a| \le 1, \ a\in \cI, \mand  |\nu_n| \le 1,\  n\in \cJ,
$$  
and $p,A,N$ satisfying
$$
  p^{1/4} \le AN \le p^{5/4},  \qquad N \le Ap^{1/4},
$$
we have,
$$
|\SanIJp |  \le \(AN^{1/2}p^{1/2}   +A^{13/16} N^{13/16}p^{43/64}\)p^{o(1)}. 
$$
\end{lemma}

For the sums $\SnIJd$ we recall the bound  from~\cite{Shp}:

\begin{lemma}
\label{lem:SnIJd} For integers  $1 \le A,N\le d$, and weights $\bnu$   with
$$
|\nu_n| \le 1, \qquad  n\in \cJ,
$$ 
 we have,
$$
|\SnIJd |   \le N^{3/4} \(A^{1/8}d   +A^{1/2}d^{3/4}\)d^{o(1)}. 
$$
\end{lemma}

Note that for $q \ge N \ge A^{3/2}$ the following bound
$$
|\SnIJd  | \le N^{1/2} A^{1/2}d^{1+o(1)}, 
$$
 from~\cite[Theorem~3.2]{ShpZha} is stronger (analysing the
proof one can see that the result holds without any changes for composite moduli).
However in our case it does not seem to bring any new results.

Furthermore, when $q=p$ is prime the bound of Blomer,  Fouvry, Kowalski, Michel 
 and Mili{\'c}evi{\'c}~\cite[Theorem~5.1(2)]{BFKMM}
 $$
|\SnIJp |  \le A^{7/12} N^{5/6} p^{3/4+o(1)}, 
$$
which holds under the conditions
$$
A, N \le p, \qquad AN \le p^{3/2},  \qquad N \le A^2,
$$
and extends the range of non-triviality 
of Lemma~\ref{lem:SnIJd} but is weaker near $A\sim N \sim p^{1/2}$ which is 
a crucial range for our argument. We note that in~\cite{BFKMM}
it is formulated only for the initial interval $\cI = \{1, \ldots, A\}$, but seems to extend 
to arbitrary intervals $\cI= \{B+1, \ldots,  B+A\}$. 

Finally, recent bounds on both $\SanIJp$  and $\SnIJp$, due to  Kowalski, Michel and 
Sawin~\cite{KMS2},  are nontrivial in wider ranges but  again apply only 
to prime  moduli $q=p$.

To conclude we stress that the aforementioned bounds from~\cite{BFKMM,FKM, KMS2}
do not seem to improve the results of Theorems~\ref{thm:DAR} and~\ref{thm:EAR} in 
interesting ranges, that is when $A\le q^{1/2}$ and  $q\ge X^{2/3}$ 
(even for prime $q = p$).

\subsection{Characters and multiplicative congruences}

Let $\cX_q$ be the set of
{\it multiplicative\/} characters of the residue ring modulo $q \ge 1$ 
and let  $\cX_q^*=\cX_q\setminus\{\chi_0\}$ be the set of   
{\it nonprincipal\/} characters; we refer the reader to~\cite[Chapter~3]{IwKow} for the relevant background.
In particular, we make use of the following orthogonality property of 
characters, see~\cite[Section~3.2]{IwKow},
\begin{equation}
\label{eq:orth}
\frac{1}{\varphi(q)}\sum_{\chi\in \cX_q} \chi(a)
 = \begin{cases} 
 1,  & \text{if}\ a \equiv 1 \pmod q,\\
0,  & \text{otherwise,}
\end{cases}
\end{equation} 
which holds for any integer $a$ with $\gcd(a,q)=1$.

We need the following result of 
Ayyad, Cochrane and Zheng~\cite[Theorem~1]{ACZ}, see also~\cite{Ke} for sharper error terms in the case
of prime $q = p$.

\begin{lemma}
\label{lem:4th-Moment}
For any integers $K$ and $H$ we have
$$
\sum_{\chi\in\cX_q^*}
\left|\sum_{x=K}^{K+H}\chi(x)\right|^4\ll H^2 q^{o(1)}.
$$
\end{lemma}
We note that there is no restriction $H<q$ in the statement of Lemma~\ref{lem:4th-Moment} and this is  important for the proof of Theorem~\ref{thm:DAR-small}.
 
\begin{lemma}
\label{lem:multcong}
Let $p$ be prime and let $\cH_i$ be intervals  of lengths $H_i$, $i=1,2,3,4$.   Then we have 
\begin{align*}
\#\{\(x_1,x_2,x_3,x_4\) & \in \cH_1 \times \cH_2 \times  \cH_3 \times \cH_4~ \\ & \quad    x_1\dots x_4\not \equiv 0 \mod{p}  :~x_1x_2\equiv x_3x_4 \pmod{p}\}\\
& \qquad \qquad \ll \frac{H_1H_2H_3H_4}{p}+O\((H_1H_2H_3H_4)^{1/2}p^{o(1)}\).
\end{align*}
\end{lemma}

\subsection{Reduction to smooth sums}

For a  a smooth function $g(x,y)$ we define the Fourier 
transform
$$
\widehat g\left(u,v\right)=\int_{\R^2}g\left(x,y\right)\e(ux +vy )dxdy,  
$$
where $\e(z) = \exp(2 \pi i z)$.

The following is~\cite[Proposition~4.11]{IwKow}.

\begin{lemma}
\label{lem:PoissonSummationTau}
Let $q$ and $z$ be integers  with $\gcd(z,q)=1$.  For a smooth function $g(x,y)$  with a compact support we define
$$
\tau_g(n)=\sum_{m_1m_2=n}g(m_1,m_2)  \mand  \tau_h(n)=\sum_{m_1m_2=n}h(m_1,m_2), 
$$
where  
$$
h(x,y)=\frac{1}{q}\widehat g\left(\frac{x}{q},\frac{y}{q}\right),
$$
and
$$
\tau_h(0)=\int_{\R^2}\left(\frac{1}{q}+\left\{\frac{x}{q}\right\}\frac{\partial}{\partial x}+\left\{\frac{y}{q}\right\}\frac{\partial}{\partial y} \right)g(x,y)dxdy. 
$$
Then we have 
$$
\sum_{m\in \Z}\tau_g(m)\eq(zm)= \sum_{\substack{n\in \Z}}\tau_h(n)\eq(-z^{-1}n).
$$
\end{lemma}

We require a variant of Lemma~\ref{lem:PoissonSummationTau} with twists by multiplicative characters.

\begin{lemma}
\label{lem:PoissonSummationTauMult}
With notation as in Lemma~\ref{lem:PoissonSummationTau}, let $\chi$ be a primitive character mod $q$. Then we have 
\begin{align*}
\sum_{m\in \Z}\tau_g(m)\chi(m)= \eta(\chi)\sum_{\substack{n\in \Z}}\tau_h(n)\overline \chi(n),
\end{align*}
where 
\begin{align*}
\eta(\chi)=\frac{\overline{\tau(\chi)}}{\tau(\chi)},
\end{align*}
and
$$\tau(\chi)=\sum_{z=1}^{q}\overline \chi(z)\eq(z),$$
 denotes the Gauss sum.
\end{lemma}
\begin{proof}
We have 
\begin{align*}
\sum_{m\in \Z}\tau_g(m)\chi(m)=\frac{1}{\tau(\chi)}\sum_{z=1}^{q-1}\overline \chi(z)\sum_{m\in \Z}\tau_g(m)\eq(zm),
\end{align*}
and hence by Lemma~\ref{lem:PoissonSummationTau}
\begin{align*}
\sum_{m\in \Z}\tau_g(m)\chi(m)&=\frac{1}{\tau(\chi)}\sum_{z=1}^{q-1}\overline \chi(z)\sum_{n\in \Z}\tau_h(n)\eq(-z^{-1}n) \\
&=\frac{1}{\tau(\chi)}\sum_{n\in \Z}\tau_h(n)\sum_{z=1}^{q-1}\overline \chi(z)\eq(-z^{-1}n),
\end{align*}
which simplifies to 
\begin{align*}
\sum_{m\in \Z}\tau_g(m)\chi(m)=\frac{\chi(-1)\tau(\overline \chi)}{\tau(\chi)}\sum_{n\in \Z}\tau_h(n)\overline \chi(n),
\end{align*}
and the result follows since 
\begin{align*}
\overline{\tau(\overline \chi)}=\chi(-1)\tau(\chi).
\end{align*}
\end{proof}
 For a proof of the following, see~\cite[Lemma~2]{Fouv}.
\begin{lemma}
\label{lem:smoothpartition}
There exists a sequence of smooth functions $\varPsi_{\ell}$,  $\ell = 0, 1, \ldots$  satisfying  
$$
\supp \varPsi_{\ell}\subseteq [2^{\ell-1},2^{\ell+1}] \mand  \varPsi_\ell^{(k)}(x)\ll \frac{1}{x^{k}},
$$
with the implied constant depending on $k,$ such that  for any $x\ge 1$ we have  
$$
\sum_{\ell\ge 0}\varPsi_{\ell}(x)=1.
$$
\end{lemma}

\section{Proofs of Main Results}
%\label{sec main}

\subsection{Proof of Theorem~\ref{thm:DAR}}

We now set
\begin{equation}
\begin{split}
\label{eq: Y large p}
Y = \max\{A^{1/2} X^{1/2+\varepsilon/2}  p^{3/8}, & A^{-1/2} X^{1/2+\varepsilon/2}  p^{7/8} , \\
& \qquad A^{-1/6} X^{5/18} p^{83/72} \}
\end{split} 
 \end{equation}
 and also 
 \begin{equation}
\label{eq: UV}
 U = U(p)= p^2 X^{-1} \mand V = V(p)= p^2 X^{1+\eps} Y^{-2}.
 \end{equation}
So, using our assumption that $p \ge X^{4/7}$,  we easily verify that 
 \begin{equation}
\label{eq: AUV}
  p^{1/4} \le AU \le AV \le p^{5/4} \mand  U \le V  \le Ap^{1/4} < p . 
 \end{equation}
We also verify that 
 \begin{equation}
\label{eq: AXYp}
A^{13/16}      X^{5/16}   Y^{-1/8} p^{19/64}  \le  AY/p.
 \end{equation}
 
We now fix some $\varepsilon > 0$ and define the  integer $\ell$   by the conditions 
$$
 2^{\ell-1} U  \le V<  U 2^{\ell}
$$
and set 
$$
V_i = \min\left\{ 2^i  U , V \right\}, \quad i = 0, \ldots, \ell.
$$
Hence, we derive from~\eqref{eq: DaR} that  for a prime $q=p$ we have
$$
\DaRp =  \frac{1}{p}  \sum_\pm  \sum_{a \in \cI}  
 \sum_{n=1}^{V(p)}   \alpha_{a,p} \tau(n)  u_p^\pm(n)  \cK_p(\mp n, a).
$$
Therefore, 
\begin{equation}
\label{eq: D1D2}
\DaRp \le  \frac{1}{p}   \sum_\pm\(D_1^{\pm} + D_2^{\pm}\),
 \end{equation}
where 
\begin{align*}
 &D_1^{\pm}  =  \left|  \sum_{a \in \cI}  
\sum_{1 \le n \le  U}   \alpha_{a,p} \tau(n)  u_p^\pm(n)  \cK_p(\mp n, a)\right|, \\
& D_2^{\pm}  = 
 \sum_{i=0}^{\ell-1}\left|  \sum_{a \in \cI}  
\sum_{V_i \le n < V_{i+1} }    \alpha_{a,p} \tau(n)  u_p^\pm(n)  \cK_p(\mp n, a)\right|. 
\end{align*}

For $D_1^{\pm} $, using the well-known bound on the divisor function (see,
for example,~\cite[Equation~(1.81)]{IwKow}) and recalling~\eqref{eq: u size},
we have 
\begin{equation}
\label{eq: tau up}
|\tau(n)  u_p^\pm(n)| \le X^{1+\eps+o(1)} p^{-1}
 \end{equation}
for $1 \le n \le U$.

Using~\eqref{eq: AUV}, we see that  Lemma~\ref{lem:SanIJp} applies
with $N = U$. Hence  after 
rescaling the weights in order to apply  the bound~\eqref{eq: tau up},  
we obtain
$$  
D_1^{\pm}  \le X^{1+\eps} p^{-1}   \(AU^{1/2}p^{1/2}   +A^{13/16} U^{13/16}p^{43/64}\)p^{o(1)}.
$$
Thus, recalling  the choice of $U$  and 
$V$ from~\eqref{eq: UV}, 
we see that 
 \begin{equation}
 \begin{split}
\label{eq: D1 bound}
D_1^{\pm} & \le  X^{1+\eps} p^{-1}   \(AX^{-1/2}p^{3/2}   +A^{13/16} X^{-13/16}p^{147/64}\)p^{o(1)}\\
 & =    \(A X^{1/2+\eps}p^{1/2}   +A^{13/16}  X^{3/16+\eps}p^{83/64}\)p^{o(1)}. 
\end{split}
\end{equation}

To estimate $D_2^{\pm}$,  we  note that by~\eqref{eq: u size} we have
$$
|\tau(n)  u_p^\pm(n)| \le
 X^{1/4} V_i^{-3/4} p^{1/2+o(1)}
$$
for $V_i \le n < V_{i+1} $.  We also recall~\eqref{eq: AUV}, hence 
by Lemma~\ref{lem:SanIJp} 
(after rescaling the weights again) 
with $N =  V_{i+1}$, as in~\eqref{eq: D1 bound},
we obtain
 \begin{align*}
D_2^{\pm} & \le  \sum_{i=0}^{\ell-1}  X^{1/4} V_i^{-3/4} p^{1/2+o(1)}   \(AV_i^{1/2}p^{1/2}   +A^{13/16} V_i^{13/16}p^{43/64}\)\\
& =  AX^{1/4} p^{1+o(1)}   \sum_{i=0}^{\ell-1}  V_i^{-1/4}
+   A^{13/16} X^{1/4}  p^{75/64+o(1)}   \sum_{i=0}^{\ell-1}  V_i^{1/16} \\
& \le  \(AX^{1/4} p U^{-1/4}
+   A^{13/16} X^{1/4}  p^{75/64}  V^{1/16}\)p^{o(1)} . 
\end{align*}

Thus, recalling the choice of $U$  and 
$V$ from~\eqref{eq: UV},  we see that 
 \begin{equation}
\label{eq: D2 bound}
D_2^{\pm}  \le  AX^{1/2} p^{1/2+o(1)}   
+   A^{13/16}      X^{5/16+\eps/16}   Y^{-1/8} p^{83/64+o(1)}. 
\end{equation}

Substituting~\eqref{eq: D1 bound} and~\eqref{eq: D2 bound}  in~\eqref{eq: D1D2} we obtain
$$
\DaRp  \le AX^{1/2} p^{-1/2+o(1)}   
+   A^{13/16}      X^{5/16+\eps/16}   Y^{-1/8} p^{19/64+o(1)}. 
$$
Hence, by~\eqref{eq: DAB Dab} and also using that  $Y/p \ge 1$, 
we obtain
\begin{align*}
\DARp  & \le AX^{1/2} p^{-1/2+o(1)}  \\
& \qquad    +   A^{13/16}      X^{5/16+\eps/16}   Y^{-1/8} p^{19/64+o(1)}  + AY^{1+\eps}p^{-1+\eps}\\
 & \ll \(AX^{1/2} p^{-1/2}  +   A^{13/16}      X^{5/16}   Y^{-1/8} p^{19/64}  + AY/p\)X^{2\eps}.
\end{align*}
Furthermore, from~\eqref{eq: AXYp} we conclude that 
$$
\DARp   \ll  (AX^{1/2} p^{-1/2}     + AY/p)X^{2\eps}.
$$
Using that $\eps>0$  is arbitrary  and recalling the choice of $Y$ in~\eqref{eq: Y large p}, we obtain
\begin{align*}
\DARp  & \le  (AX^{1/2} p^{-1/2}     +A^{3/2} X^{1/2}  p^{-5/8} \\
&\qquad \qquad + A^{1/2} X^{1/2}  p^{-1/8} +  A^{5/6} X^{5/18} p^{11/72})p^{o(1)}, 
\end{align*}
which concludes the proof.

\subsection{Proofs of Theorem~\ref{thm:EAR}}

We now set
\begin{equation}
\label{eq: Y large}
Y = \sqrt{q X^{1+\varepsilon}},  
 \end{equation}
 so recalling~\eqref{eq: Ud Vd}, we see that we always have 
 \begin{equation}
\label{eq: U V small}
U(d) \le V(d) \le d 
 \end{equation}
for every $d \mid q$. 

We fix some $\varepsilon > 0$.
For each  positive $d \mid q$ we define the  integer $\ell(d)$   by the conditions 
$$
 2^{\ell(d)-1} U(d)  \le V(d) <  U(d) 2^{\ell(d)}
$$
and set 
$$
V_i(d) = \min\left\{ 2^i  U(d) , V(d) \right\}, \quad i = 0, \ldots, \ell(d).
$$
Hence, we derive from~\eqref{eq: EaR} that 
\begin{equation}
\label{eq: E1E2}
\EaR =  \frac{1}{q} \sum_{d \mid q} \sum_\pm\(E_1^{\pm}(d) + E_2^{\pm}(d)\),
 \end{equation}
where 
\begin{align*}
 &E_1^{\pm}(d)  =  \left|  \sum_{a \in \cI}  
\sum_{1 \le n \le  U(d)}   \tau(n)  u_d^\pm(n)  \cK_d(\mp n, a)\right|, \\
& E_2^{\pm}(d)  = 
 \sum_{i=0}^{\ell(d)-1}\left|  \sum_{a \in \cI}  
\sum_{V_i(d) \le n < V_{i+1}(d) }   \tau(n)  u_d^\pm(n)  \cK_d(\mp n, a)\right|. 
\end{align*}

For $E_1^{\pm}(d) $, from the well-known bound on the divisor function (see,
for example,~\cite[Equation~(1.81)]{IwKow}) and recalling~\eqref{eq: u size},
we have 
$$
|\tau(n)  u_d^\pm(n)| \le X^{1+\eps+o(1)} d^{-1}
$$
for $1 \le n \le U(d)$.

Using~\eqref{eq: U V small}, we see that  Lemma~\ref{lem:SnIJd} applies
with $N = U(d)$. Hence  after 
rescaling the weights, we obtain
\begin{align*}
E_1^{\pm}(d) & \le X^{1+\eps+o(1)} d^{-1}  U(d)^{3/4} \(A^{1/8}d   +A^{1/2}d^{3/4}\) \\
& = X^{1+\eps+o(1)}  U(d)^{3/4} \(A^{1/8}   +A^{1/2}d^{-1/4}\). 
\end{align*}
Thus recalling the definition of $U(d)$ in~\eqref{eq: Ud Vd}, 
we see that for any $d \mid q$ we have 
 \begin{equation}
\label{eq: E1 bound}
E_1^{\pm}(d)  \le X^{1/4+ \eps}  \(A^{1/8}q^{3/2}   +A^{1/2}q^{5/4} \)q^{o(1)}.
\end{equation}
 
To estimate $E_2^{\pm}(d)$,  we  note that by~\eqref{eq: u size} we have
$$
|\tau(n)  u_d^\pm(n)| \le
 d^{1/2} X^{1/4} V_i(d)^{-3/4} q^{o(1)}
$$
for $V_i(d) \le n < V_{i+1}(d)$.  We also recall~\eqref{eq: U V small}, hence 
by Lemma~\ref{lem:SnIJd} 
(after rescaling the weights again) 
with $N =  V_{i+1}(d)$, as in~\eqref{eq: E1 bound},
we obtain
 \begin{equation}
 \begin{split}
\label{eq: E2 bound}
E_2^{\pm}(d) & \le d^{1/2} X^{1/4}   \(A^{1/8}d   +A^{1/2}d^{3/4}\) q^{o(1)}\\
& \le X^{1/4+ \eps}  \(A^{1/8}q^{3/2}   +A^{1/2}q^{5/4} \)q^{o(1)}.
\end{split}
\end{equation}

Substituting~\eqref{eq: E1 bound} and~\eqref{eq: E2 bound}  in~\eqref{eq: E1E2} we obtain
$$
\EaR  \le X^{1/4+ \eps}  \(A^{1/8}q^{1/2}   +A^{1/2}q^{1/4} \)q^{o(1)}. 
$$
Hence, recalling~\eqref{eq: EAB Eab} and also using that  $Y/q \ge 1$, 
we obtain
$$
\EAR  \ll X^{1/4+ \eps}  \(A^{1/8}q^{3/2}   +A^{1/2}q^{5/4} \)q^{o(1)} +  AY^{1+\eps} q^{-1 + \eps}. 
$$
Using that $\eps>0$  is arbitrary and recalling the choice of $Y$ in~\eqref{eq: Y large}, we 
obtain
$$
\EAR  \le  \( A^{1/8}  X^{1/4} q^{3/2} + A^{1/2}  X^{1/4} q^{5/4}  +  A X^{1/2}   q^{-1/2 }\) q^{o(1)}, 
$$
which concludes the proof. 

\subsection{Proof of Theorem~\ref{thm:DAR-small}}

Let $U$ and $V$ be as in~\eqref{eq: UV}  and apply~\eqref{eq:RXaqComplete} with 
\begin{equation}
\label{eq:DAR-smallY}
Y=\frac{X^{1/3}p}{A^{1/3}},
\end{equation}
 to obtain
\begin{equation}
\begin{split}
\label{eq:thmDARe1}
 \cD(X;\cA,p)  &\ll \frac{1}{p}\sum_{\pm}\sum_{a\in \cA}\left|\sum_{n=1}^{\infty}\tau(n)u_p^{\pm}(n)\cK_p(\mp n,a) \right|\\
 & \qquad \qquad\qquad\qquad\qquad\qquad +A\left(\frac{Y}{p}+1 \right)(qY)^{\varepsilon} \\
&=\frac{1}{p}\left(S^++S^-\right)+A\left(\frac{Y}{p}+1 \right)(pY)^{\varepsilon},
\end{split}
\end{equation}
where 
$$
S^{\pm}=\sum_{a\in \cA}\left|\sum_{n=1}^{\infty}\tau(n)u_p^{\pm}(n)\cK_p(\mp n,a)\right|.
$$
We consider $S^-$, a similar argument applies to $S^+$. 
 With notation as in 
Lemma~\ref{lem:smoothpartition}, we have 
\begin{align*}
S^-&=\sum_{a\in \cA}\left|\sum_{m,n=1}^{\infty}u_p^{-}(mn)\cK_p(mn,a)\right| \\
& =\sum_{a\in \cA}\left|\sum_{m,n=1}^{\infty}u_p^{-}(mn)\cK_p(mn,a) \sum_{j,k=0}^{\infty}\varPsi_{j}(m)  \varPsi_{k}(n) \right|\\
&\le \sum_{j,k=0}^{\infty}\sum_{a\in \cA}\left|\sum_{m,n\in \Z}\varPsi_{j}(m)  \varPsi_{k}(n) u_p^{-}(mn)\cK_p(mn,a)\right|.
\end{align*}  
We note that we extend the summation over $m$ and $n$ to $\Z$ to be able to apply  Lemma~\ref{lem:PoissonSummationTau}
(since the functions $\varPsi_{\ell}$ are supported only on positive integers this does not change the sum). 
Hence 
$$
S^-\ll \sum_{j,k=0}^{\infty}S^-(j,k),
$$
 where 
$$
S^-(j,k)=\sum_{a\in \cA}\left|\sum_{m,n\in \Z}\varPsi_{j}(m)  \varPsi_{k}(n) u_p^{-}(mn)\cK_p(mn,a)\right|,
$$
and we recall that for each integer $\ell \ge 0,$   $\varPsi_{\ell}$ is a smooth function satisfying
\begin{equation}
\label{eq:Vellreminder}
\supp \varPsi_{\ell}\subseteq [2^{\ell-1},2^{\ell+1}] \mand  \varPsi_{\ell} (x)\ll 1.
\end{equation}

Let
\begin{equation}
\label{eq:Qdef}
Q=\max\left\{U,\frac{p}{A}\right\},
\end{equation}
 define the integer $s$ by 
$$
2^{s-1}Q\le V\le 2^{s}Q,
$$
and set 
$$
V_i=\min\{2^{i}Q,V\}, \qquad i = 0,  \ldots, s.
$$ 
Recalling~\eqref{eq:DAR-smallY}, we note that 
$$
U\le Q\le V.
$$
We partition
\begin{equation}
\label{eq:S-V}
S^-\ll \sum_{\substack{j,k\\ 2^{j+k}\le Q}} S^-(j,k)+
\sum_{i=0}^{s}\sum_{\substack{j,k=0 \\ V_i<2^{j+k}\le 2V_i}}^\infty S^-(j,k)+\sum_{\substack{j,k=0\\ 2^{j+k}\ge 2V}}^\infty S^-(j,k).
\end{equation}
By~\eqref{eq:ur3}, as in the derivation of~\eqref{eq:Rerror}, the contribution from 
the terms with $2^{j+k}\ge 2V$ is negligible, and in particular we have 
$$
\sum_{\substack{j,k=0 \\ 2^{j+k}\ge 2V}}^\infty S^-(j,k)=O(1),
$$
which after substitution to~\eqref{eq:S-V}  gives
\begin{equation}
\label{eq:Si}
S^-\ll \sum_{\substack{j,k=0 \\ 2^{j+k}\le Q}}^\infty S^-(j,k)+\sum_{i=0}^{s}\sum_{\substack{j,k=0 \\ V_i<2^{j+k}\le 2V_i}}^\infty S^-(j,k)+O(1).
\end{equation} 
We first estimate the contribution to $S^-$  from terms with $2^{j+k}\le Q$. With $U=U(q)$ as in~\eqref{eq: Ud Vd}, suppose first that $2^{j+k}\le U$. 
By~\eqref{eq: u size}, \eqref{eq:Vellreminder} 
and the Weil bound for Kloosterman sums, 
see for example~\cite[Theorem~11.11]{IwKow},  
we have 
$$
S^{-}(j,k)\ll \sum_{a\in \cA}\sum_{m,n\in \Z}| \varPsi_{j}(m) | | \varPsi_{k}(n)| u_p^{-}(mn)||\cK_p(mn,a)| \ll Ap^{3/2}X^{\varepsilon}.
$$
If $U\le 2^{j+k}\le Q$ then
\begin{align*}
S^{-}(j,k)&\ll \sum_{a\in \cA}\sum_{m,n\in \Z}| \varPsi_{j}(m) | | \varPsi_{k}(n)| |u_p^{-}(mn)||\cK_p(mn,a)| \\ & \ll 2^{(j+k)/4}AX^{1/4}p^{1+\varepsilon}\ll A Q^{1/4}X^{1/4}p^{1+\varepsilon},
\end{align*}
and hence by~\eqref{eq:Si}
\begin{equation}
\label{eq:Si1}
S^{-}\ll  Ap^{3/2+\varepsilon}(\log p)^2+Q^{1/4}AX^{1/4}p^{1+\varepsilon}+\sum_{i=0}^{s}\sum_{\substack{j,k \\ V_i<2^{j+k}\le 2V_i}} S^-(j,k).
\end{equation} 
 Fix some integer $t\le s$, some pair $(j,k)$ satisfying
\begin{equation}
\label{eq:Vellj}
V_t<2^{j+k}\le 2V_t,
\end{equation}
and consider $S^-(j,k)$. Expanding the Kloosterman sum and interchanging summation, we have 
$$
S^-(j,k)=\sum_{a\in \cA}\left|\sum_{y=1}^{p-1}\ep(ay^{-1})\sum_{m,n\in \Z}\varPsi_{j}(m)  \varPsi_{k}(n) u_p^{-}(mn)\ep(mny)\right|.
$$
Applying Lemma~\ref{lem:PoissonSummationTau} with 
\begin{equation}
\label{eq:gxydef}
g(x,y)=\varPsi_{j}(m)  \varPsi_{k}(n) u_p^{-}(mn),
\end{equation}
gives 
$$
S^{-}(j,k)=\sum_{a\in \cA}\left|\sum_{n\in \Z}\tau_h(n)\sum_{y=1}^{p-1}\ep((a-n)y^{-1})\right|,
$$
with notation as in Lemma~\ref{lem:PoissonSummationTau} and $h$ is given by
$$
h(x,y)=\frac{1}{p}\widehat g\left(\frac{x}{p},\frac{y}{p}\right).
$$
For some sequence of complex numbers $\vartheta(a)$ satisfying $|\vartheta(a)|=1$ we have 
\begin{equation}
\label{eq:S-kj}
S^-(j,k)=\sum_{y=1}^{p-1}\sum_{a\in \cA}\sum_{n\in \Z}\tau_h(n)\vartheta(a)\ep((a-n)y^{-1}),
\end{equation}
and hence 
$$
S^-(j,k)=p\sum_{\substack{a\in \cA, n\in \Z \\ n\equiv a \pmod p}}\vartheta(a)\tau_h(n)-\sum_{a\in \cA}\vartheta(a)\sum_{n\in \Z}\tau_h(n).
$$
Since $0\not \in \cA$ and $p$ is prime, using~\eqref{eq:orth} to control  
the condition $n\equiv a \pmod p$ via multiplicative characters, we have 
$$
\sum_{\substack{a\in \cA, n\in \Z \\ n\equiv a \pmod p}}\vartheta(a)\tau_h(n)=\frac{1}{p-1}\sum_{\chi}\left(\sum_{a\in \cA}\vartheta(a)\overline \chi(a) \right)\left(\sum_{n\in \Z}\tau_h(n)\chi(n) \right),
$$
and isolating the contribution from the trivial character gives
\begin{align*}
& \sum_{\substack{a\in \cA, n\in \Z \\ n\equiv a \pmod p}}\vartheta(a)\tau_h(n)=\frac{1}{p-1}\sum_{\substack{\chi\neq \chi_0}}\left(\sum_{a\in \cA}\vartheta(a)\overline \chi(a) \right)\left(\sum_{n\in \Z}\tau_h(n)\chi(n) \right) \\ 
& \quad \quad \quad \quad +\frac{1}{p-1}\sum_{a\in \cA}\vartheta(a)\sum_{\substack{n\in \Z}}\tau_h(n)-\frac{1}{p-1}\sum_{a\in \cA}\vartheta(a)\sum_{\substack{n\in \Z \\ n\equiv 0 \pmod{p}}}\tau_h(n).
\end{align*}
Hence by~\eqref{eq:S-kj}
\begin{equation}
\label{eq:S-jk1}
S^{-}(j,k)\ll \Sigma_1+\frac{1}{p}\Sigma_2+\Sigma_3,
\end{equation}
where 
\begin{align*}
& \Sigma_1=\sum_{\substack{\chi\neq \chi_0}}\left(\sum_{a\in \cA}\vartheta(a)\overline \chi(a) \right)\left(\sum_{n\in \Z}\tau_h(n)\chi(n) \right),\\
& \Sigma_2=\sum_{a\in \cA}\vartheta(a)\sum_{n\in \Z}\tau_h(n),\\
& \Sigma_3=\sum_{a\in \cA}\vartheta(a)\sum_{\substack{n\in \Z \\ n\equiv 0 \pmod{p}}}\tau_h(n).
\end{align*}
Considering first $\Sigma_2$, we have 
$$
\sum_{n\in \Z}\tau_h(n)=\frac{1}{p}\sum_{m,n\in \Z}\widehat g\left(\frac{m}{p},\frac{n}{p}\right).
$$
With $g$ as in~\eqref{eq:gxydef} we have 
\begin{equation}
\label{eq:Poisson-p}
\sum_{m,n\in \Z}\widehat g\left(\frac{m}{p},\frac{n}{p}\right)=p^2\sum_{m,n\in \Z}g(pm,pn) . 
\end{equation} \
Indeed, to see this it is enough to notice that for $G(x,y) = g(px, py)$
we have 
\begin{align*}
\widehat G\left(u,z\right) &=\int_{\R^2}G\left(x,y\right)\e(ux +zy )dxdy \\
& =  \frac{1}{p^2}  \int_{\R^2}g\left(px,py\right)\e((u/p) (px) +(v/p)(py) )d(px)d(py)\\
& =  \frac{1}{p^2}  \int_{\R^2}g\left(x,y\right)\e((u/p) x +(v/p) y )dxdy =
\ \frac{1}{p^2} \widehat g\left(\frac{m}{p},\frac{n}{p}\right), 
\end{align*}
and we obtain~\eqref{eq:Poisson-p} by Poission summation.  Hence, from the equations~\eqref{eq:gxydef} and~\eqref{eq:Poisson-p}, we see that
$$
\sum_{n\in \Z}\tau_h(n)=p\sum_{m,n\in \Z}\Psi_j(pm)\Psi_k(pn)u_p^{-}(p^2mn).
$$
By~\eqref{eq: u size} and~\eqref{eq:Vellj} 
\begin{align*}
\sum_{m,n\in \Z}|\Psi_j(pm)||\Psi_k(pn)| &| u_p^{-}(p^2mn)|\\
&\ll \sum_{2^{j-1}/p\le m\le 2^{j+1}/p}\sum_{2^{k-1}/p\le m\le 2^{k+1}/p}|u_p^{-}(p^2mn) |\\ & \ll \frac{2^{j+k}}{p^2}\frac{X^{1/4}p^{1/2}}{2^{3(i+j)/4}}\ll  X^{1/4}V_{t}^{1/4} p^{-3/2},
\end{align*}
so that 
\begin{equation}
\label{eq:tauhn}
\sum_{n\in \Z}\tau_h(n)\ll  X^{1/4}V_{t}^{1/4}p^{-1/2},
\end{equation}
which implies
\begin{equation}
\label{eq:Sigma2b}
\Sigma_2\ll  AX^{1/4}V_{t}^{1/4}p^{-1/2}.
\end{equation}
Considering $\Sigma_3$, by~\eqref{eq:tauhn} we have 
\begin{equation}
\begin{split}
\label{eq:Sigma31}
 \sum_{\substack{n\in \Z \\ n\equiv 0 \pmod{p}}}&\tau_h(n)\\
 &=\frac{1}{p}\sum_{z=1}^{p-1}\sum_{n\in \Z} \tau_h(n)\ep(zn)+\frac{1}{p}\sum_{n\in \Z}\tau_h(n) \\
&=\frac{1}{p}\sum_{z=1}^{p-1}\sum_{n\in \Z}\tau_h(n)\ep(zn)+O\(X^{1/4}V_{t}^{1/4} p^{-3/2}\).
\end{split}
\end{equation}
By Lemma~\ref{lem:PoissonSummationTau} and~\eqref{eq:gxydef}
$$
\sum_{n\in \Z}\tau_h(n)\ep(zn)=\sum_{n\in \Z}\tau_g(n)\ep(-z^{-1}n),
$$
so that
\begin{align*}
\sum_{z=1}^{p-1}\sum_{n\in \Z}\tau_h(n)\ep(zn)&=\sum_{n\in \Z}\tau_g(n)\sum_{z=1}^{p-1}\ep(-z^{-1}n) \\ & \ll p\sum_{\substack{n\in \Z \\ n\equiv 0 \pmod{p}}}\tau_g(n)+\sum_{n\in \Z}\tau_g(n).
\end{align*}
We have 
\begin{align*}
\sum_{\substack{n\in \Z \\ n\equiv 0 \pmod{p}}}\tau_g(n)& \ll \sum_{n,m\in \Z}
| \varPsi_{j}(pm) | | \varPsi_{k}(n)| |u_p^{-}(pmn)|
\\ & \qquad \quad +\sum_{n,m\in \Z}
| \varPsi_{j}(m) | | \varPsi_{k}(pn)| |u_p^{-}(pmn)|  . 
\end{align*}
For the first sum above
\begin{align*}
\sum_{n,m\in \Z}\Psi_j(pm)\Psi_k(n)u_p^{-}(pmn)&\ll \sum_{2^{j-1}/p\le m\le 2^{j+1}/p}\sum_{2^{k-1}\le n\le 2^{k+1}}|u_p^{-}(pmn)| \\
& \ll \frac{2^{j+k}}{p}\frac{X^{1/4}p^{1/2}}{2^{3(i+j)/4}}\ll X^{1/4}V_t^{1/4} p^{-1/2},
\end{align*}
and similarly for the second sum
$$
\sum_{n,m\in \Z}
| \varPsi_{j}(m) | | \varPsi_{k}(pn)| |u_p^{-}(pmn)|  \ll  X^{1/4} V_t^{1/4}p^{-1/2} .
$$
A similar argument shows 
$$
\sum_{n\in \Z}\tau_g(n)\ll X^{1/4}V_t^{1/4} p^{1/2},
$$
so that 
$$
\sum_{z=1}^{p-1}\sum_{n\in \Z}\tau_h(n)\ep(zn)\ll X^{1/4}p^{1/2}V_t^{1/4}.
$$
Hence by~\eqref{eq:Sigma31}
\begin{equation}
\label{eq:Sigma3b}
\Sigma_3\ll AX^{1/4}V_t^{1/4}p^{-1/2}.
\end{equation}

It now remains to estimate $\Sigma_1$. B Lemma~\ref{lem:PoissonSummationTauMult} we have
\begin{equation}
\label{eq:Sigma1step1}
\Sigma_1\le \sum_{\chi\neq \chi_0}\left|\sum_{a\in \cA}\vartheta(a)\overline \chi(a)\right|\left|\sum_{n\in \Z}\tau_g(n)\overline \chi(-n) \right|,
\end{equation}
and hence by the Cauchy-Schwarz inequality 
\begin{align*}
\Sigma_1^2&\le \sum_{\chi}\left|\sum_{a\in \cA}\vartheta(a)\overline \chi(a)\right|^2\sum_{\chi}\left|\sum_{n\in \Z}\tau_g(n)\overline \chi(-n) \right|^2 \\
&\le p^2A\sum_{\substack{m,n\in \Z \\ \gcd(mn,p)=1 \\ n\equiv m \pmod{p}}}\tau_g(m)\tau_g(n).
\end{align*}
We have 
$$
\sum_{\substack{m,n\in \Z \\ \gcd(mn,p)=1 \\ n\equiv m \pmod{p}}}\tau_g(m)\tau_g(n)\ll \sum_{\substack{ 2^{j-1}\le m_1,m_2\le 2^{j+1} \\ 2^{k-1} \le n_1,n_2\le 2^{k+1} \\  n_1m_1\equiv n_2m_2 \pmod{p} \\ \gcd(m_1   m_2 n_1 n_2,p)=1}}|u_p^{-}(m_1n_1)||u^{-}_p(m_2n_2)|.
$$
By~\eqref{eq: u size} and Lemma~\ref{lem:multcong}
\begin{align*}
\sum_{\substack{m,n\in \Z \\ \gcd(mn,p)=1 \\ n\equiv m \pmod{p}}}\tau_g(m)\tau_g(n)&\le \frac{X^{1/2}p}{2^{3(i+j)/2}}\left(\frac{2^{2(i+j)}}{p}+2^{i+j}\right)p^{o(1)}  \\ &
\le \left(X^{1/2}V_t^{1/2}+\frac{X^{1/2}p}{V_t^{1/2}}\right)p^{o(1)},
\end{align*}
and hence by~\eqref{eq:Sigma1step1}
$$
\Sigma_1\le  A^{1/2}\left(X^{1/4}V_t^{1/4}+\frac{X^{1/4}p^{1/2}}{V_t^{1/4}}\right)p^{1+o(1)}.
$$

Since $A<p$, from~\eqref{eq:S-jk1},~\eqref{eq:Sigma2b} and~\eqref{eq:Sigma3b}
$$
S^{-}(j,k)\le A^{1/2}\left(X^{1/4}V_t^{1/4}+\frac{X^{1/4}p^{1/2}}{V_t^{1/4}}\right)p^{1+o(1)},
$$
and hence by~\eqref{eq:Si1} 
\begin{align*}
S^{-}&\ll  Ap^{3/2+\varepsilon+o(1)}+Q^{1/4}AX^{1/4}p^{1+\varepsilon}  \\ & \qquad \qquad  +\sum_{i=0}^{s}\sum_{\substack{j,k \\ V_i<2^{j+k}\le 2V_i}} A^{1/2}\left(X^{1/4}V_t^{1/4}+\frac{X^{1/4}p^{1/2}}{V_t^{1/4}}\right)p^{1+o(1)}.
\end{align*} 
 Since the above sum   contains $O((\log X)^3)$ terms, in the negative powers of $V_t$ we replace $V_t$ 
 with its smallest possible value $Q$, while in  the positive  powers of $V_t$ we replace $V_t$ 
 with its largest  possible value $V$ and derive
\begin{align*}
S^{-}\ll Ap^{3/2+\varepsilon+o(1)}+Q^{1/4}AX^{1/4}p^{1+\varepsilon}+A^{1/2}X^{1/4}&V^{1/2}p^{1+o(1)} \\ &  +\frac{A^{1/2}X^{1/4}}{Q^{1/4} p^{3/2+o(1)}}.
\end{align*}
Recalling~\eqref{eq:Qdef}  we see that $Q\ge p/A$ thus 
$$
\frac{A^{1/2}X^{1/4}}{Q^{1/4} p^{3/2}}  \le \frac{A^{1/2}X^{1/4}p^{3/2+o(1)}}{ (p/A)^{1/4}}
= A^{3/4}X^{1/4}p^{5/4} .
$$
Moreover
\begin{align*}
Q^{1/4}AX^{1/4}p \le U^{1/4}AX^{1/4}p + A^{3/4}X^{1/4}p^{5/4} = Ap^{3/2}   + A^{3/4}X^{1/4}p^{5/4}. 
\end{align*}  
Therefore, we obtain 
$$
S^-\ll Ap^{3/2+\varepsilon+o(1)}+A^{3/4}X^{1/4}p^{5/4+\varepsilon}+A^{1/2}X^{1/4}V^{1/4}p^{1+o(1)}.
$$
A similar estimate holds for $S^{+}$ and hence by~\eqref{eq: Ud Vd} and~\eqref{eq:thmDARe1}  
\begin{align*}
 \cD(X;\cA,p) \ll Ap^{1/2+\varepsilon+o(1)}+A^{3/4}X^{1/4}p^{1/4+\varepsilon}& +\frac{A^{1/2}X^{1/2}p^{1/2+o(1)}}{Y^{1/2}} \\ &   +A\left(\frac{Y}{p}+1 \right)(pY)^{\varepsilon},
\end{align*}
which by~\eqref{eq:DAR-smallY} implies
$$
 \cD(X;\cA,p) \le \left(Ap^{1/2}+A^{3/4}X^{1/4}p^{1/4}+A^{2/3}X^{1/3}\right)(pX)^{\varepsilon+o(1)}.
$$
Since $\varepsilon>0$ is arbitrary, we obtain
$$
 \cD(X;\cA,p) \le \left(Ap^{1/2}+A^{3/4}X^{1/4}p^{1/4}+A^{2/3}X^{1/3}\right)X^{o(1)}.
$$
Clearly, if $p> X/A$ them $Ap^{1/2} > A^{1/2} X^{1/2}$ and thus the bound~\eqref{eq: Cauchy} is stronger.
On the other hand, for $p \le X/A$ we have $Ap^{1/2} \le A^{3/4}X^{1/4}p^{1/4}$, which 
this completes the proof.

\section{Remarks}

For almost all $q$ a stronger version of 
Lemma~\ref{lem:SnIJd} has also been given in~\cite{Shp}.
In turn, this can be used to improve Theorem~\ref{thm:EAR}
for almost all moduli $q$. In fact, limiting this set of moduli to 
only prime $q=p$ simplifies this question significantly. For composite 
values of $q$ one also has to eliminate $q$ having an ``undesirable'' 
divisor $d \mid q$, yet there is little doubt that this can be done. 

It is also interesting to extend Theorem~\ref{thm:DAR} to arbitrary 
integer moduli $q$. Unfortunately, the analogues of the bounds 
on bilinear sums $\SanIJd$ from~\cite{BFKMM} are known only for prime $q$,
while the method of~\cite{Shp,ShpZha} seems to work only for the sums
$\SnIJd$.

\section*{Acknowledgement}

This work  was partially supported by the  Australian Research Council 
Grant DP170100786.

\end{document}